\documentclass[a4paper,12pt]{amsart}
\usepackage{amsmath, amssymb, amsthm, xcolor}
\usepackage{verbatim}
\usepackage{xypic}
\usepackage{graphicx}
\usepackage{enumerate}
\usepackage[colorlinks=true, citecolor=blue, linkcolor=blue, bookmarks=true]{hyperref}

\usepackage{tikz}

\usetikzlibrary{intersections}

\newcounter{dummy}

\newtheorem{theorem}[dummy]{Theorem}
\newtheorem{lemma}[dummy]{Lemma}

\newtheorem{proposition}[dummy]{Proposition}

\newtheorem*{theorem*}{Theorem}

\theoremstyle{definition}
\newtheorem{definition}[dummy]{Definition}

\DeclareMathOperator{\id}{id}

\newcommand{\dimnuc}{\mathrm{dim}_\mathrm{nuc}}

\newcommand{\prim}{\mathrm{Prim}}

\newcommand{\N}{\mathbb{N}}
\newcommand{\M}{\mathbb{M}}
\newcommand{\C}{\mathbb{C}}
\newcommand{\K}{\mathbb{K}}
\renewcommand{\O}{\mathcal{O}}
\newcommand{\I}{\mathcal{I}}

\begin{document}

\title{Nuclear dimension of extensions of $\O_\infty$-stable algebras}
\author{Samuel Evington}
	
\address{Mathematical Institute, University of M{\"u}nster, Einsteinstrasse 62, 48149 M{\"u}nster, Germany}
\email{evington@uni-muenster.de}

\thanks{The project was partially supported by the Deutsche Forschungsgemeinschaft (DFG, German Research Foundation) – Project-ID 427320536 – SFB 1442, as well as under Germany’s Excellence Strategy EXC 2044 390685587, Mathematics M{\"u}nster: Dynamics–Geometry–Structure, and also partially supported by ERC Advanced Grant 834267 - AMAREC - and by EPSRC grant EP/R025061/2.}

\maketitle

\begin{abstract}
	We obtain an improved upper bound for the nuclear dimension of extensions of $\O_\infty$-stable $\rm{C}^*$-algebras. In particular, we prove that the nuclear dimension of a full extension of an $\O_\infty$-stable $\rm{C}^*$-algebra by a stable AF algebra is one.
\end{abstract}

\section*{Introduction}

Nuclear dimension is a dimension theory for $\rm{C}^*$-algebras, introduced by Winter and Zacharias (\cite{WZ10}), defined in terms of coloured finite-dimensional approximations. In the commutative case, it reduces to a version of Lebesgue covering dimension, whereby a compact space is said to have dimension at most $n$ if every open cover has a refinement that can be $(n+1)$-coloured so that sets with the same colour are disjoint (\cite{Os65}). 

Nuclear dimension has close connections to the K-theoretic classification of separable $\rm{C}^*$-algebras. The zero dimensional objects are the approximately finite dimensional (AF) algebras, whose classification by Elliott via the $K_0$ group was a key motivation for the subsequent classification programme for nuclear $\rm{C}^*$-algebras (\cite{El76, El95}). Moreover, finite nuclear dimension has emerged as a fundamental regularity property for simple, separable, nuclear $\rm{C}^*$-algebras (\cite{ET08}), which separates the class of classifiable $\rm{C}^*$-algebras from pathological counter-examples (\cite{To08,Ro03,Vi98}). This has now crystallised in the classification theorem for simple, separable, unital, $\rm{C}^*$-algebras of finite nuclear dimension satisfying the universal coefficient theorem (\cite{Kir95,Phi00,Go15, EGLN15, TWW17}).

Determining the exact value for the nuclear dimension of a C$^*$-algebra can be a challenge, as it requires finding an approximation using the minimal possible number of colours. However, there has been recent progress in the case of simple C$^*$-algebras (\cite{CETWW}), $\O_\infty$-stable $\rm{C}^*$-algebras (\cite{BGSW19}) and the Toeplitz algebra (\cite{BW19}). The latter has led to renewed interest in the nuclear dimension of extensions (\cite{Gla20}).

Given a $\rm{C}^*$-algebra $A$, it is known that the nuclear dimension of every ideal $J \lhd A$ and every quotient $A / J$ is bounded by that of $A$. In special cases, one can show that the nuclear dimension of $A$ is the maximum of $\dimnuc(J)$ and $\dimnuc(A/J)$. This happens in both the commutative setting and the approximately finite-dimensional setting, and when the extension $0\rightarrow J \rightarrow A \rightarrow A / J\rightarrow 0$ is particularly nice, for example a direct sum or a quasidiagonal extension (\cite{WZ10,RSS15}).
However, in the fully general setting, the best known upper bound is given by
\begin{equation*}
	\dimnuc(A)+1 \leq (\dimnuc(A/J)+1) + (\dimnuc(J)+1),
\end{equation*}
which comes from directly combining an $(n+1)$-coloured approximation to the ideal with an $(m+1)$-coloured approximation to the quotient (\cite[Proposition 2.9]{WZ10}).

The main result of this paper is an improved upper bound for the nuclear dimension of a large family of extensions.

\begin{theorem}\label{thm:main}
	Let $0\rightarrow J \rightarrow A \rightarrow B\rightarrow 0$ be a full extension where $J$ is a separable stable $\rm{C}^*$-algebra and $B$ is a separable, nuclear, $\O_\infty$-stable $\rm{C}^*$-algebra. Then
\begin{equation*}
	1 \leq \dimnuc(A) \leq \dimnuc(J) + 1.
\end{equation*}
In particular, if $J$ is a stable approximately finite-dimensional algebra, then $\dimnuc(A) = 1$.
\end{theorem}

Recall that an extension $0\rightarrow J \rightarrow A \rightarrow B\rightarrow 0$  is called \emph{full} if the Busby invariant $\beta:B \rightarrow Q(J)$ is a full $^*$-homomorphism, i.e.\ $\beta(b)$ is norm full in the corona algebra $Q(J)$ for every non-zero $b \in B$. (See Section \ref{sec:Busby} for additional definitions.)  When $B$ is simple, all unital extensions are full. So Theorem \ref{thm:main} generalises \cite[Theorem 1]{Gla20}. Another 
special case of note is when $J$ is the compact operators $\K$. In this case, $Q(J)$ is simple, so any injective $^*$-homomorphism $B \rightarrow Q(J)$ is full. 
More generally, the ideals of $Q(J)$ correspond to ideals of the multiplier algebra $M(J)$ containing $J$, and the ideal structure of $M(J)$ for AF algebras is known (\cite{El74, El87, Ro91}).

A full discussion of Ext-theory is beyond the scope of this paper. However, we note that whenever $J$ is stable and $B$ is separable there will be a full trivial extension, so every class in $\mathrm{Ext}(B, J)$ will contain full representatives (see Section \ref{sec:Busby}). This gives many examples where Theorem \ref{thm:main} applies. 

The proof of Theorem \ref{thm:main} builds on the strategy developed in \cite{Gla20}, based on approximations of the form
\begin{align}\label{eqn:intro-approx}
	a &\approx [(1-h_{i+1})\mu(k_i\pi(a)) + (h_i-h_{i+1}) \mu(\pi(a))] \tag{$\dagger$} \\
	  &\quad \quad + (1-h_{i+1})\mu((1-k_i)\pi(a))  \nonumber \\
	  &\quad \quad + h_i a, \nonumber 
\end{align} 
where $\pi:A \rightarrow B$ is the quotient map, $\mu:B \rightarrow A$ is a c.p.c.\ splitting, $(k_i)_i$ is a approximately central sequence in $B$, and $(h_i)_i$ is a quasicentral approximate unit with $h_{i+1}h_i = h_i$.  The final term is in the ideal and the rest factor through the quotient. Ultimately, it is the approximate orthogonality between the final two terms of \eqref{eqn:intro-approx} that allows for the improved dimension estimate.

Constructing the finite dimensional approximations to $A$ is substantially harder in the generality of Theorem \ref{thm:main} than in the special case considered in \cite{Gla20}. Since $B$ is not assumed to be simple, we must use the ideal-related version of Gabe's uniqueness theorem for nuclear, $\O_\infty$-stable maps (\cite{Ga19}). Gabe's uniqueness theorem is used to relate the approximations in \eqref{eqn:intro-approx} to a system of finite-dimensional approximations constructed by Bosa, Gabe, Sims and White in their proof that separable, nuclear, $\O_\infty$-stable $\rm{C}^*$-algebras have nuclear dimension one (\cite{BGSW19}). At a more technical level, we make essential use of Elliott and Kucerovsky's notion of \emph{purely large} extensions (\cite{El01}) to prove some structural results about ideals in the extension $\rm{C}^*$-algebra and its ultrapower. The extension in Theorem \ref{thm:main} will be purely large whenever $\dimnuc(J) < \infty$ by results of Robert (\cite{Ro11}) and Kucerovsky--Ng (\cite{KN06}).

\subsection*{Structure of paper}
In Section \ref{sec:prelims}, we fix some global notational conventions and recap some preliminary material for the benefit of the reader. In Section \ref{sec:purely-large}, we review purely large extensions and prove some technical lemmas needed in the proof of the main theorem. In Section \ref{sec:uniqueness}, we review actions of a topological space on a $\rm{C}^*$-algebra and isolate the uniqueness result used in the proof of the main theorem. In Section \ref{sec:nuclear-dimension}, we prove Theorem \ref{thm:main}. 

\subsection*{Acknowledgements} This paper would not have been possible without the Glasgow Summer Project 2019 on the Cuntz--Teoplitz algebras. I would like to thank everyone involved. I would also like to thank Christian B{\"o}nicke for discussions that informed this paper, Jorge Castillejos and Stuart White for their comments on an earlier version of this paper, Jamie Gabe for suggesting Proposition \ref{prop:PureInf-PureLarge}, and Antoine le Calvez for his talk on the Elliott--Kucerovsky theorem, where I learnt about purely large extensions. 

\section{preliminaries}\label{sec:prelims}

\subsection{Notation}

We write $\M_n$ for the $\rm{C}^*$-algebra of $n \times n$ matrices, $\K$ for the $\rm{C}^*$-algebra of compact operators, and $\O_n$ for the Cuntz algebra with $n$ generators (including $n = \infty$). We write $A_+$ for the positive cone of a $\rm{C}^*$-algebra $A$, and we write $a \approx_\epsilon b$ as a shorthand for $\|a - b\| \leq \epsilon$. We use the abbreviations c.p.\ maps and c.p.c.\ maps for completely positive maps and completely positive contractive maps, respectively. 

We write $A^\sim$ for the unitisation of the $\rm{C}^*$-algebra $A$. (For avoidance of ambiguity, we fix our convention to be that $A^\sim \cong A \oplus \C$ when $A$ is already unital.) We will often work in the unitisation for ease of notation and during calculations. For example, we shall consider the hereditary subalgebras of a potentially non-unital $\rm{C}^*$-algebra $A$ of the form $\overline{(1-a)A(1-a)}$ for $a \in A_+$. We write $M(A)$ for the multiplier algebra of $A$ and $Q(A)$ for the corona algebra $M(A)/A$.

We fix a free ultrafilter $\omega$ on $\N$ that will be used throughout the paper. Given a sequence of $\rm{C}^*$-algebras $(A_i)_i$, we write $\prod_{i \in \N} A_i$ for the \emph{product} $\rm{C}^*$-algebra, i.e.\ the $\rm{C}^*$-algebra of norm bounded sequences $(a_i)_i$ with $a_i \in A_i$. We then define the \emph{ultraproduct} to be the quotient
\begin{equation}
	\prod_{i \to \omega} A_i = \frac{\prod_{i \in \N} A_i}{\lbrace (a_i)_i: \lim_{i \to \omega}\|a_i\| = 0 \rbrace}.
\end{equation} 
The \emph{ultrapower} of a $\rm{C}^*$-algebra $A$ is the ultraproduct of the constant sequence with $A_i = A$ for all $i \in \N$. We write $A_\omega$ for the ultrapower of $A$ and identify $A$ with the subalgebra of $A_\omega$ coming from constant sequences. We write $A_\omega \cap A'$ for the commutant of $A$ in $A_\omega$. The ideal of $A_\omega$ given by $\mathrm{Ann}(A) = \{x \in A_\omega: xa = 0 = ax \mbox{ for all } a \in A\}$  is called the annihilator of $A$ in $A_\omega$.   

Given a map $\phi:A\rightarrow B$, we shall write $\phi \oplus 0$ to denote the map $A\rightarrow \M_2(B)$ given by
\begin{equation}
	a \mapsto 
\begin{pmatrix}
\phi(a) & 0\\
0  & 0 
\end{pmatrix}.
\end{equation}

\subsection{Order zero maps and nuclear dimension}

Let $A$ and $B$ be $\rm{C}^*$-algebras. A c.p.\ map $\phi:A \rightarrow B$ is \emph{order zero} if it preserves orthogonality, i.e.\ $\phi(a)\phi(b) = 0$ whenever $a,b \in A_+$ satisfy $ab = 0$. There is a one-to-one correspondence between c.p.c.\ order zero maps $\phi:A \rightarrow B$ and $^*$-homomorphisms from the cone $\Phi:C_0(0,1] \otimes A \rightarrow B$ where $\phi(a) = \Phi(\mathrm{id}_{(0,1]} \otimes a)$ for all $a \in A$; see \cite[Corollary 3.1]{WZ09}. When $F$ is a finite dimensional $\rm{C}^*$-algebra, the cone $C_0(0,1] \otimes F$ is projective by \cite[Theorem 4.9]{Lo93}. This leads to following lifting result for c.p.c.\ order zero maps with finite dimensional domains.

\begin{proposition}[{\cite[Proposition 1.2.4]{Wi09}}]\label{prop:oz-lifting}
Let $\pi:A \rightarrow B$ be a surjective $^*$-homomorphism of $\rm{C}^*$-algebras. Suppose $\phi:F \rightarrow B$ is a c.p.c.\ order zero map with $F$ a finite dimensional $\rm{C}^*$-algebra. Then there exists a c.p.c.\ order zero map $\hat{\phi}:F \rightarrow A$ such that $\phi = \pi \circ \hat{\phi}$.
\end{proposition}

We now recall the definition of nuclear dimension of a $\rm{C}^*$-algebra and more generally the nuclear dimension of a $^*$-homomorphism.
\begin{definition}[{c.f.\ \cite[Definition 2.1]{WZ10}}]\label{def:DimNuc}
A $^*$-homomorphism of $\rm{C}^*$-algebras $\pi:A \to B$ is said to have \emph{nuclear dimension} at most $n$, denoted $\dim_{\mathrm{nuc}}(\pi) \leq n$, if for every finite subset $\mathcal{F} \subseteq A$ and every $\epsilon > 0$, there exist $n+1$ finite-dimensional $\rm{C}^*$-algebras $G^{(0)}, \ldots, G^{(n)}$ together with functions $\psi: A \rightarrow G^{(0)} \oplus \ldots \oplus G^{(n)}$ and $\phi:G^{(0)} \oplus \ldots \oplus G^{(n)} \rightarrow B$ such that $\psi$ is a c.p.c.\ map, each restriction $\phi|_{G^{(j)}}$ is a c.p.c.\ order zero map and, for every $a \in \mathcal{F}$,
\begin{equation} \|\pi(a) - \phi(\psi(a))\| < \epsilon. \end{equation} 

The \emph{nuclear dimension} of a $\rm{C}^*$-algebra $A$, denoted $\dim_{\mathrm{nuc}}(A)$, is then defined as the nuclear dimension of the identity homomorphism $\id_A$.
\end{definition}

\subsection{Ideals and hereditary subalgebras}\label{sec:ideals}

In this paper, \emph{ideal} will always mean a two-sided closed ideal. We write $\I(A)$ for the partially ordered set of ideals of a $\rm{C}^*$-algebra $A$. Given a $^*$-homomorphism $f:A \rightarrow B$, we define $\I(f): \I(A) \rightarrow \I(B)$ by $I \mapsto \overline{Bf(I)B}$. Every ideal in a hereditary subalgebra $B \subseteq A$ is of the form $I \cap B$ for some ideal $I \lhd A$; see for example \cite[Theorem 3.2.7]{Mu90}. We record a useful corollary of this fact. 

\begin{proposition}\label{prop:hereditary}
	Let $B$ be a hereditary subalgebra of the  $\rm{C}^*$-algebra $A$. Let $S \subseteq B$. Suppose $I \lhd A$ is the ideal in $A$ generated by $S$. Then $I \cap B$ is the ideal in $B$ generated by $S$.  
\end{proposition}
\begin{proof}
	Since $I \lhd A$, we have $I \cap B \lhd B$. By hypothesis, $S \subseteq  I \cap B$. Suppose that $J \lhd B$ and $S \subseteq J$. As $B$ is a hereditary subalgebra of $A$, there is $K \lhd A$ such that $J = K \cap B$. Then $S \subseteq K$. Hence, $I \subseteq K$ because $I$ is generated by $S$. Therefore, $I \cap B \subseteq J$.
\end{proof}

\subsection{\texorpdfstring{Purely infinite $\rm{C}^*$-algebras}{Purely infinite C*-algebras}}\label{sec:purely-infinite}

Let $A$ be a $\rm{C}^*$-algebra and let $a,b \in A_+$. We recall that $a$ is said to be \emph{Cuntz subequivalent} to $b$, denoted by $a \precsim b$, if there exists a sequence $(r_i)_i$ in $A$ such that $a = \lim_{i\to\infty}r_i^* b r_i$. Furthermore, we recall that a \emph{character} on a $\rm{C}^*$-algebra is a non-zero $^*$-homomorphism $A \rightarrow  \C$.

Following \cite{KR00}, a $\rm{C}^*$-algebra $A$ is said to be \emph{purely infinite} if there are no characters on $A$ and for any $a,b \in A_+$, $a \precsim b$ whenever $a$ is in the ideal generated by $b$. Every $\O_\infty$-stable $\rm{C}^*$-algebra is purely infinite by \cite[Proposition 4.5]{KR00}. Moreover, purely infiniteness passes to ideals, quotients, products and ultraproducts; see \cite[Proposition 4.5 and Proposition 4.20]{KR00}.

\subsection{Extensions and the Busby invariant}\label{sec:Busby}

An \emph{extension} of $\rm{C}^*$-algebras is a short exact sequence $0 \rightarrow J \xrightarrow{\iota} A \xrightarrow{\pi} B \rightarrow 0$. We will identify $J$ with the ideal $\iota(J) \lhd A$ and $B$ with the quotient $A / \iota(J)$. 

We now briefly recall the definition of the \emph{Busby invariant} of an extension $0 \rightarrow J \rightarrow A \rightarrow B \rightarrow 0$; for a more expansive account, we refer the reader to \cite[Section 15.2]{Bla86}.  As $J \lhd A$, every element $a \in A$ defines an element of the multiplier algebra  $M(J)$ via multiplication. The resulting $^*$-homomorphism $A \rightarrow M(J)$ extends the canonical embedding of $J$ into $M(J)$, so there is an induced $^*$-homomorphism from $B = A / J$ to $Q(J) = M(J)/J$. This $^*$-homomorphism $\beta:B \rightarrow Q(J)$ is the Busby invariant of the extension. 

An extension can be completely recovered from its Busby invariant via a pull back construction. Indeed, writing $q:M(J) \rightarrow Q(J)$ for the quotient map, we have $A \cong \{(x,b) \in M(J) \oplus B: q(x) = \beta(b) \}$; see \cite[Example 15.3.2]{Bla86}. In particular, we have that $A$ is unital if and only if $\beta$ is unital, and $J$ is an essential ideal in $A$ if and only if $\beta$ is injective. 

We now recall the definition of a (norm) full extension from \cite[Definition 1(iii)]{KN06}. 

\begin{definition}
	An extension $0\rightarrow J \rightarrow A \rightarrow B\rightarrow 0$ is \emph{full} if its Busby invariant $\beta:B \rightarrow Q(J)$ is a full $^*$-homomorphism, i.e.\ whenever $\beta(b)$ is norm full in $Q(J)$ for all non-zero $b \in B$.
\end{definition}

When $B$ is simple, it suffices to check that a single non-zero element has full image under $\beta$. In particular, unital extensions will be full whenever $B$ is simple.
When $Q(J)$ is simple, we see that all essential extensions are full.
This occurs when $J = \K$ and when $J = C \otimes \K$ where $C$ is unital, simple and purely infinite by \cite[Theorem 4.2]{Ro91}.

Suppose $B$ is separable and $J = J_0 \otimes \K$ is stable. Fix a faithful representation $\rho:B \rightarrow B(\ell^2)$. Replacing $\rho$ with its infinite direct sum, we may assume that $\rho$ is full. Then $\tau:B \rightarrow M(J_0) \otimes_{\mathrm{alg}}  B(\ell^2) \subseteq M(J)$ defined by $b \mapsto 1 \otimes \rho(b)$ is full. Composing with the quotient map $M(J) \rightarrow Q(J)$, we get a full extension $\bar{\tau}:B \rightarrow Q(J)$. Moreover, if $\beta:B \rightarrow Q(J)$ is any other extension, then $\beta \oplus \bar{\tau}:B \rightarrow \M_2(Q(J)) \cong Q(J)$ is a full extension. 

In Ext-theory, the extension $\bar{\tau}$ is said to be \emph{trivial}, as it has a lift $\tau$. The extensions $\beta \oplus \bar{\tau}$ and $\beta$ are said to be \emph{stably equivalent} and define the same element of the group $\mathrm{Ext}(B,J)$; see \cite[Definition 16.3]{Bla86}. 
Therefore, every element of $\mathrm{Ext}(B,J)$ has a full representative.

\section{Purely large extensions}\label{sec:purely-large}

The notion of a purely large extension was introduced by Elliott and Kucerovsky in \cite{El01}, motivated by their work on absorbing representations and Ext theory. We recall their definition here.

\begin{definition}{(\cite[Definition 1]{El01})}
	An extension $0\rightarrow J \rightarrow A \rightarrow B\rightarrow 0$ is \emph{purely large} if for any $a \in A \setminus J$ the hereditary subalgebra $\overline{aJa^*}$ contains a stable, $\sigma$-unital subalgebra $D$ that is full in $J$.
\end{definition}

The following elementary result ensures that certain cut-downs of purely large extensions remain purely large.
\begin{proposition}\label{prop:cutdown}
	Suppose	$0\rightarrow J \rightarrow A \rightarrow B\rightarrow 0$ is a purely large extension. For any $h \in J_+$, the induced extension 
\begin{equation}\label{eqn:cutdown}
0\rightarrow \overline{(1-h)J(1-h)} \rightarrow \overline{(1-h)A(1-h)} \rightarrow B\rightarrow 0
\end{equation}
is purely large.
\end{proposition}
\begin{proof}
Let $a \in \overline{(1-h)A(1-h)} \setminus \overline{(1-h)J(1-h)}$. Then $a(1-h) \in A \setminus J$. Since $0\rightarrow J \rightarrow A \rightarrow B\rightarrow 0$ is a purely large extension, the hereditary subalgebra $\overline{a(1-h)J(1-h)a^*}$ contains a stable, $\sigma$-unital subalgebra $D$ that is full in $J$. 
Since $D \subseteq \overline{(1-h)J(1-h)}$ and $\overline{(1-h)J(1-h)}$ is a hereditary subalgebra of $J$, $D$ is full in $\overline{(1-h)J(1-h)}$ by Proposition \ref{prop:hereditary}. Moreover, $\overline{a(1-h)J(1-h)a^*} = \overline{a\overline{(1-h)J(1-h)}a^*}$. Therefore, the extension \eqref{eqn:cutdown} is purely large.
\end{proof}

Next, we recall an approximation property of purely large extensions, established by Elliott and Kucerovsky, that will play a crucial role in this paper.   
\begin{proposition}[{\cite[Lemma 7]{El01}}]\label{prop:EK}
	Suppose	$0\rightarrow J \rightarrow A \rightarrow B\rightarrow 0$ is a purely large extension. Let $a \in A \setminus J$. Then for any $x \in J$ and $\epsilon > 0$, there exists $x_0 \in J$ such that $\|x - x_0ax_0^*\| \leq \epsilon$. 
Moreover, if $\|a + J \| = 1$ and $\|x\| = 1$, then $x_0$ can be chosen with $\|x_0\| = 1$.
\end{proposition}

Proposition \ref{prop:EK} implies that every ideal of $A$ either contains $J$ or is contained in $J$. The norm bounds in Proposition \ref{prop:EK} ensure that a similar result is true after passing to ultraproducts. This is the subject of the following proposition, which generalises \cite[Proposition 4(i)]{Gla20}. 

\begin{proposition}\label{lem:6AM}
Let	$(0\rightarrow J_i \rightarrow A_i \xrightarrow{\pi_i} B_i \rightarrow 0)_i$ be a sequence of purely large extensions where $B_i$ are $\O_\infty$-stable $\rm{C}^*$-algebras.
Let 
\begin{equation}\label{eqn:ultraproduct-extension}
   0 \rightarrow \prod_{i\to\omega}J_i \rightarrow \prod_{i\to\omega}A_i \xrightarrow{\pi_\omega} \prod_{i\to\omega}B_i \rightarrow 0
\end{equation}
be the induced extension of the corresponding ultrapowers. Let $x,y \in A_\omega$. Then $x$ is in the ideal generated by $y$ whenever $\pi_\omega(x)$ is in the ideal generated by $\pi_\omega(y)$.
\end{proposition}
\begin{proof}
	Without loss of generality, we may assume $x$ and $y$ are positive and non-zero. Suppose $\pi_\omega(x)$ is in the ideal generated by $\pi_\omega(y)$. Let $\epsilon > 0$. As each $B_i$ is $\O_\infty$-stable, $\prod_{i\to\omega}B_i$ is purely infinite; see Section \ref{sec:purely-infinite}. Since $\pi_\omega(x)$ is in the ideal generated by $\pi_\omega(y)$ and $\prod_{i\to\omega}B_i$ is purely infinite, we have $\pi_\omega(x) \precsim \pi_\omega(y)$. By \cite[Lemma 2.2]{KR02}, there exists $b \in \prod_{i\to\omega}B_i$ such that $(\pi_\omega(x)-\epsilon)_+ = b\pi_\omega(y)b^*$. Let $a \in \prod_{i\to\omega}A_i$ be a lift of $b$, and set $z := aya^* - (x-\epsilon)_+$. As \eqref{eqn:ultraproduct-extension} is exact,  $z \in \prod_{i\to\omega}J_i$.

Say $y = (y_i)_i$ for $y_i \in A_i$ and $z = (z_i)_i$ for $z_i \in J_i$. We may assume without loss of generality that $\|y_i + J_i\| \geq \|\pi_\omega(y)\| > 0$  and $\|z_i\| \leq \|z\|$ for all $i \in \N$. By Proposition \ref{prop:EK}, there exist $w_i \in A$ such that $\|z_i - w_iy_iw_i^*\| \leq \epsilon$ and $\|w_i\|^2 \leq \|z\|\|\pi_\omega(y)\|^{-1}$ for all $i \in \N$. (This norm estimate for $w_i$ comes from scaling $y_i$ and $z_i$ so that $\|y_i + J_i \| = 1$ and $\|z_i\| = 1$ before applying Proposition \ref{prop:EK}.) Let $w = (w_i)_i \in \prod_{i\to\omega}A_i$. Then
\begin{equation}
	aya^* - wyw^* \approx_\epsilon (x-\epsilon)_+	\approx_\epsilon x.
\end{equation}
Hence, $x$ is in the ideal generated by $y$. 
\end{proof}

We now turn to sufficient conditions for an extension to be purely large. 
To this end, we recall the definition of the \emph{corona factorisation property}, due to Kucerovsky.

\begin{definition}
	A $\rm{C}^*$-algebra $J$ has the \emph{corona factorisation property} if every full projection in $M(J \otimes \K)$ is Murray--von Neumann equivalent to $1_{M(J \otimes \K)}$.
\end{definition}

The corona factorisation property is a fairly weak regularity property for $\rm{C}^*$-algebras and is implied by finite nuclear dimension.

\begin{proposition}[{\cite[Corollary 3.5]{Ro11}}]\label{prop:NucDimCFP}
	Suppose $J$ has finite nuclear dimension. Then $J$ has the corona factorisation property. 
\end{proposition}

We now record a sufficient condition for an extension to be purely large, namely fullness combined with the corona factorisation property for the ideal.  This result is due to Kucerovsky--Ng and can be extracted from the proof of \cite[Theorem 3.5, $(v) \Rightarrow (i)$]{KN06}. For the benefit of the reader, we also include this argument.

\begin{proposition}[{\cite{KN06}}]\label{prop:CFPpurelylarge}
	Suppose $J$ is stable, $\sigma$-unital and has the corona factorisation property. Then every full extension $0\rightarrow J \rightarrow A \rightarrow B\rightarrow 0$ is purely large.  
\end{proposition}
\begin{proof}
	Let $a \in A \setminus J$. Without loss of generality, we may assume $a \geq 0$. Fullness of the extension means that the element $a$ defines a full element of $Q(J)$. Hence, $a$ is full in $M(J)$ by \cite[Theorem 3.3]{KN06}. By \cite[Theorem 3.2]{KN06}, we have $\overline{aJa} \cong pJp$ for some full projection $p \in M(J)$. By the corona factorisation property, $p$ is Murray--von Neumann equivalent to $1$ in $M(J)$. Hence,  $pJp \cong J$ is stable and $\sigma$-unital. Taking $D = \overline{aJa}$, we see that the extension is purely large. 
\end{proof}

Propositions \ref{prop:NucDimCFP}  and \ref{prop:CFPpurelylarge} ensure that the extension that appears in Theorem \ref{thm:main} is purely large when $\dimnuc(J) < \infty$.
The following proposition, communicated to the author by Gabe, is an alternative argument for proving that this extension is purely large, using pure infiniteness of the quotient instead of the corona factorisation property of the ideal.
\begin{proposition}\label{prop:PureInf-PureLarge}
	Suppose $B$ is purely infinite and that $J$ is stable and $\sigma$-unital. Then every full extension $0\rightarrow J \rightarrow A \xrightarrow{\pi} B\rightarrow 0$ is purely large.  
\end{proposition}
\begin{proof}
	View $A \subseteq M(J)$ and $B \subseteq Q(J)$. Let $a \in A_+ \setminus J$. As $\pi(a^2)$ is full in $Q(J)$, there exists $N \in \N$ such that $1_{Q(J)} \precsim \pi(a^2)^{\oplus N}$. By \cite[Theorem 4.16]{KR00}, all non-zero elements in $B$ are properly infinite, so $\pi(a^2)^{\oplus N} \precsim \pi(a^2)$ in $B \otimes \K$ (and also in $Q(J) \otimes \K$). 
	
Since $1_{Q(J)} \precsim \pi(a^2)$, there is $x \in Q(J)$ such that $x^*\pi(a^2)x = 1_{Q(J)}$. Let $y \in M(J)$ be a lift of $x$. Then $y^*a^2y - 1_{M(J)} \in J$. As $J$ is stable, we have $\|v^*(y^*a^2y - 1_{M(J)})v\| < 1$ for some isometry $v \in M(J)$. Since $v^*y^*a^2yv$ is invertible, we have $z^*a^2z = 1_{M(J)}$ where $z = y v(v^*y^*a^2yv)^{-1/2}$. Set $D = \overline{azJz^*a} \subseteq \overline{aJa}$. Then $z^*aDaz = J$. It follows that $D$ is full in $J$. Moreover, $D$ is isomorphic to $J$ because $az$ is an isometry, so $D$ is $\sigma$-unital and stable. Therefore, the extension is purely large.
\end{proof}

\section{The uniqueness theorem}\label{sec:uniqueness}

A key tool in the proof of Theorem \ref{thm:main} is a uniqueness theorem that will provide sufficient conditions for two $^*$-homomorphisms $C_0(0,1] \otimes B \rightarrow A_\omega$ to be unitary equivalent (after passing to a matrix inflation). This uniqueness theorem is deduced from Gabe's work on the classification of $\O_\infty$-stable maps \cite{Ga19}, which reproves and extends Kirchberg's results on the classification of $\O_\infty$-stable algebras \cite{Ki00}. We begin by reviewing the necessary definitions related to actions of topological spaces on $\rm{C}^*$-algebras.    

\subsection{\texorpdfstring{$X$-$\rm{C}^*$-algebras}{X-C*-algebras}}

Given a topological space $X$, write $\O(X)$ for the complete lattice of open subsets of $X$. An \emph{action} of $X$ on a $\rm{C}^*$-algebra $A$ is an order-preserving map $\O(X) \rightarrow \I(A)$. We write $A(U)$ for the image of the open set $U$ under the action. An \emph{$X$-$\rm{C}^*$-algebra} is a $\rm{C}^*$-algebra endowed with an action of $X$.

An $X$-action on a C$^*$-algebra $A$ is said to be \emph{lower-semicontinuous} if the map $\O(X) \rightarrow \I(A)$ preserves infima. In particular, this means that for every $a \in A$ there is a smallest open set $U_a$ such that $a \in A(U_a)$. (In fact, lower-semicontinuity is equivalent to this statement, but we will not need the converse; see \cite[Remark 10.19]{Ga19}.) An $X$-$\rm{C}^*$-algebra is said to be lower-semicontinuous if the action of $X$ is lower-semicontinuous.

Suppose $f:A \rightarrow B$ is a $^*$-homomorphism between $X$-$\rm{C}^*$-algebras. Then $f$ is \emph{$X$-equivariant} if $f(A(U)) \subseteq B(U)$ for all open sets $U$. Suppose that $A$ is a lower-semicontinuous $X$-$\rm{C}^*$-algebra.   Then an $X$-equivariant $^*$-homomorphism $f:A \rightarrow B$ is said to be \emph{$X$-full} if $f(a)$ is full in $B(U_a)$ for every $a \in A$.

Every $\rm{C}^*$-algebra $B$ can be endowed with the canonical action of its primitive ideal space $\prim(B)$. Indeed, the open subsets of $\prim(B)$ are in order-preserving one-to-one correspondence with the ideal of $B$ via the map sending $U$ to the intersection of all primitive ideals in $\prim(B) \setminus U$; see \cite[Theorem 4.1.3]{Ped79}. This action is clearly lower-semicontinuous.

Given an $X$-$\rm{C}^*$-algebra $B$, there is an induced $X$-action on the cone $C_0(0,1] \otimes B$ given by $U \mapsto C_0(0,1] \otimes B(U)$. When endowed with this $X$-action, the cone $C_0(0,1] \otimes B$ is \emph{$X$-equivariantly contractible}, i.e.\ the identity map $\mathrm{id}_{C_0(0,1] \otimes B}$ is homotopic to the zero map via $X$-equivariant $^*$-homomorphisms. There is also an induced $X$-action on the ultrapower $B_\omega$ given by $U \mapsto \overline{B_\omega B(U) B_\omega}$.

Given a hereditary subalgebra $C \subseteq B$, there is an induced $X$-action on $C$ given by $U \mapsto C \cap B(U)$. By Proposition \ref{prop:hereditary}, the co-restriction of an $X$-full $^*$-homomorphism $f:A \rightarrow B$ to a hereditary subalgebra $C \subseteq B$ containing the image of $f$ is an $X$-full $^*$-homomorphism  $A \rightarrow C$.

Given a $^*$-homomorphism $\pi:A \rightarrow B$, an $X$-action on $B$ can be pulled back to an $X$-action on $A$,  by defining $A(U) = \pi^{-1}(B(U))$. Viewed now as a map between $X$-$\rm{C}^*$-algebras, the $^*$-homomorphism $\pi$ is $X$-equivariant by construction. 

\subsection{Uniqueness}

We are now ready to state the required uniqueness theorem, the proof of which is an application of \cite[Theorem F]{Ga19}.

\begin{proposition}\label{thm:OinfClassification}
Let $(A_i)_i$ be a sequence of $\rm{C}^*$-algebras. Let $B$ a separable, nuclear, $\O_\infty$-stable $\rm{C}^*$-algebra. Let $X$ be a topological space acting on $\prod_{i\to \omega}A_i$ and acting lower-semicontinuously on $B$. Equip $C_0(0,1] \otimes B$ with the induced $X$-action.

Suppose $\Phi, \Psi:C_0(0,1] \otimes B \rightarrow \prod_{i\to \omega}A_i$ are $X$-full $^*$-homomorphisms. Then there exists a unitary $u \in \M_{2}(\prod_{i\to \omega}A_i)^\sim$ such that
  \begin{align}
        \Psi(x)\oplus 0 &= u(\Phi(x)\oplus 0)u^*,\quad x\in C_0(0,1]\otimes B.
         \end{align} 

\end{proposition}
\begin{proof}
	Firstly, we replace the codomain $\prod_{i\to \omega}A_i$ with a $\sigma$-unital hereditary subalgebra. Indeed, Since $C_0(0,1] \otimes B$ is separable, there exists a positive contraction $e \in \prod_{i\to\omega}A_i$ that acts as a unit on $\mathrm{Im}(\Phi)$ and $\mathrm{Im}(\Psi)$; see for example \cite[Lemma 1.16]{BBSTWW}. Let $C$ be the hereditary subalgebra of $\prod_{i\to\omega}A_i$ generated by $e$, and equip $C$ with the induced $X$-action. By construction, $C$ is $\sigma$-unital, and the co-restrictions of the maps $\Phi$ and $\Psi$ remain $X$-full by Proposition \ref{prop:hereditary}.
		
We now check that the remaining conditions of \cite[Theorem F]{Ga19} are satisfied: Firstly, $C_0(0,1] \otimes B$ is a separable and exact because $B$ is a separable and nuclear. Secondly, as the action of $X$ on $B$ is lower-semicontinuous, so is the induced action on $C_0(0,1] \otimes B$. Thirdly, in addition to being $X$-full, the maps $\Phi$ and $\Psi$ are nuclear and strongly $\mathcal{O}_{\infty}$-stable because $B$ is nuclear and $\O_\infty$-stable. Finally, $ C_{0}(0,1]\otimes B$ is $X$-equivariantly contractible, so the $^*$-homomorphisms $\Phi$ and $\Psi$ are $X$-equivariantly homotopic to zero, and so $KK(X; \Phi)=KK(X; \Psi)=0$. Therefore, by \cite[Theorem F]{Ga19}, $\Phi$ and $\Psi$ are asymptotically Murray--von Neumann equivalent.
 
Using \cite[Proposition 3.10]{Gabe}, we deduce that $\Phi \oplus 0$ and $\Psi \oplus 0$ are asymptotically unitarily equivalent with the unitaries in the minimal unitisation $\M_2(C)^\sim \subseteq \M_{2}(\prod_{i\to \omega}A_i)^\sim$. As $C_0(0,1] \otimes B$ is separable, asymptotic unitary equivalence of maps from $C_0(0,1] \otimes B$ into an ultraproduct implies unitary equivalence by Kirchberg's $\epsilon$-test (\cite[Lemma~A.1]{Kir06}). Hence, there exists a unitary $u \in \M_2(\prod_{i\rightarrow \omega} A_i)^\sim $ such that for all $x \in C_0(0,1]\otimes B$, $ \Psi(x)\oplus 0 = u(\Phi(x)\oplus 0)u^*$.   
\end{proof}

\subsection{\texorpdfstring{Testing $X$-fullness}{Testing-X-fullness}}

In order to apply Proposition \ref{thm:OinfClassification}, we need some machinery to test whether certain $X$-equivariant maps between $X$-$\rm{C}^*$-algebras are $X$-full. The following proposition shows that for maps $C_0(0,1] \otimes B \rightarrow B_\omega$ it suffices to understand the behaviour on ideals of the form $I_{C_0(0,1]} \otimes I_B$. Recall that $\I(\Phi)$ denotes the induced map on ideals, as defined in Section \ref{sec:ideals}. 

\begin{proposition}\label{prop:fullness}
Let $B$ be a $\rm{C}^*$-algebra and write $X = \prim(B)$. Endow $B$ with the canonical $X$-action and let $C_0(0,1] \otimes B$ and $B_\omega$ have the induced $X$-actions.

Suppose $\Phi:C_0(0,1] \otimes B \rightarrow B_\omega$ is $^*$-homomorphism such that
\begin{equation}\label{eqn:ideals}
\I(\Phi)(I_{C_0(0,1]} \otimes I_B) = \overline{B_\omega I_B B_\omega}
\end{equation}
for any non-zero ideals $I_{C_0(0,1]} \lhd C_0(0,1]$ and $I_B \lhd B$. Then $\Phi$ is $X$-full
\end{proposition}
\begin{proof}
	Let $f \in C_0(0,1] \otimes B$. Let $U_f \in \O(X)$ be the smallest open set with $f \in C_0(0,1] \otimes B(U_f)$. Let $I \lhd C_0(0,1] \otimes B$ be the ideal generated by $f$. Since $C_0(0,1]$ is exact, we have that  $I = \overline{\sum_{\lambda \in \Lambda} I^{(\lambda)}_{C_0(0,1]} \otimes I^{(\lambda)}_{B}}$ for non-zero ideals $I^{(\lambda)}_{C_0(0,1]} \lhd C_0(0,1]$ and $I^{(\lambda)}_{B} \lhd B$; see for example \cite[Corollary 4.6]{Br08}. Then $B(U_f) = \overline{\sum_{\lambda \in \Lambda} I^{(\lambda)}_{B}}$ and, by \eqref{eqn:ideals}, we have 
\begin{equation}
	\I(\Phi)(I)	= \overline{\sum_{\lambda \in \Lambda} \overline{B_\omega I^{(\lambda)}_{B} B_\omega}} = \overline{B_\omega B(U_f) B_\omega}.
\end{equation}
Hence, $\Phi$ is $X$-full. 
\end{proof} 

For maps into purely large extensions of $\O_\infty$-stable algebras, $X$-fullness can be tested in the quotient. This is also true for maps into ultraproducts of purely large extensions of $\O_\infty$-stable algebras. We record this in the following proposition. 

\begin{proposition}\label{prop:fullness6AM}
Let	$(0\rightarrow J_i \rightarrow A_i \xrightarrow{\pi_i} B_i \rightarrow 0)_i$ be a sequence of purely large extensions. Let $X$ be a topological space acting on $\prod_{i \to \omega} B_i$. Equip $\prod_{i \to \omega} A_i$  with the induced $X$-action obtained by pulling back ideals via $\pi_\omega$.

Let $D$ be a lower-semicontinuous $X$-$\rm{C}^*$-algebra and $\Phi:D \rightarrow \prod_{i \to \omega} A_i$ be a $^*$-homomorphism. Then $\Phi$ is $X$-full whenever $\pi_\omega \circ \Phi$ is $X$-full.
\end{proposition}
\begin{proof}
	Let $d \in D$. Let $U_d \in \O(X)$ be the smallest open set with $d \in D(U_d)$. Let $I \lhd \prod_{i \to \omega} B_i$ be the ideal corresponding to the open set $U_d$. Let $x \in \pi_\omega^{-1}(I)$. Then $\pi_\omega(x) \in I$, which is the ideal generated by $\pi_\omega(\Phi(d))$ since $\pi_\omega \circ \Phi$ is $X$-full. By Proposition \ref{lem:6AM}, we see that $x$ is in the ideal generated by $\Phi(d)$. Therefore, $\Phi$ is $X$-full.
\end{proof}

\section{Nuclear Dimension}\label{sec:nuclear-dimension}

The final ingredient needed for the proof of Theorem \ref{thm:main} are the c.p.c.\ approximations for $\O_\infty$-stable algebras constructed in \cite{BGSW19}, which respect the ideal structure of $B$. We isolate this as a lemma for the benefit of the reader and to fix notation.

\begin{lemma}[{cf. \cite[Lemma 3.5]{BGSW19}}]\label{lem:approximations}
	Let	$B$ be a separable, $\O_\infty$-stable $\rm{C}^*$-algebra. There exists a sequence of c.p.c.\ maps $\phi_i:B \rightarrow B$, which factorise through finite dimensional $\rm{C}^*$-algebras $F_i$ as $\phi_i = \eta_i \circ \psi_i$ for c.p.c.\ maps $\psi_i:B \rightarrow F_i$ and c.p.c.\ order zero maps $\eta_i:F_i \rightarrow B$, such that
\begin{enumerate}[(i)]
	\item The induced map $\psi=(\psi_i)_i:B \rightarrow \prod_{i \to \omega}F_i$ is a c.p.c.\ order zero map.
	\item The induced map $\phi=(\phi_i)_i:B \rightarrow B_\omega$ is a c.p.c.\ order zero map, and the corresponding $^*$-homomorphism $\Phi:C_0(0,1] \otimes B \rightarrow B_\omega$ satisfies 
\begin{equation}
	\I(\Phi)(I_{C_0(0,1]} \otimes I_B) = \overline{B_\omega I_B B_\omega}
\end{equation}
for any non-zero ideals $I_{C_0(0,1]} \lhd C_0(0,1]$ and $I_B \lhd B$.
\end{enumerate}
\end{lemma}
\begin{proof}
	The result follows by taking $A = B$ and $\Theta:\I(B)\rightarrow \I(B)$ to be the identity in \cite[Lemma 3.5]{BGSW19}. We alert the reader to a small conflict of notation: In \cite[Lemma 3.5]{BGSW19}, $\psi_i$ are constructed as approximate $^*$-homomorphisms $C_0(0,1] \otimes B \rightarrow F_i$, so one must compose with $b \mapsto \mathrm{id}_{(0,1]} \otimes b$ to get approximate order zero maps $B \rightarrow F_i$. We also note that \cite{BGSW19} uses $B_{(\infty)} := \ell^\infty(B)/c_0(B)$ instead of the ultrapower $B_\omega$. However, the identity on $\ell^\infty(B)$ induces a surjective $^*$-homomorphism $q:B_{(\infty)} \rightarrow B_\omega$ with $\I(q)(\overline{B_{(\infty)} I_B B_{(\infty)}}) = \overline{B_\omega I_B B_\omega}$, so the ultrapower version follows.
\end{proof}

We now proceed with the proof of Theorem \ref{thm:main}. 
\begin{proof}[Proof of Theorem \ref{thm:main}]
We have $\dimnuc(A) \geq \dimnuc(B) = 1$ by \cite[Proposition 2.3]{WZ10} and \cite[Theorem A]{BGSW19}, so we focus on the upper bound.
There is nothing to prove in the case that $\dimnuc(J)$ is infinite, so we assume hereinafter that $n := \dimnuc(J) < \infty$. Let $\pi:A \rightarrow B$ denote the quotient map of the extension. Fix a c.p.c.\ splitting $\mu:B \rightarrow A$, which exists by the Choi-Effros lifting theorem \cite{Choi-Effros}. Write $\pi_\omega:A_\omega \rightarrow B_\omega$ and $\mu_\omega:B_\omega \rightarrow A_\omega$ for the induced maps at the level of the ultrapowers.

Fix a unital embedding $\iota_{\O_\infty}:\O_\infty \hookrightarrow (B_\omega \cap B') / \mathrm{Ann}(B)$, which exists as $B$ is $\O_\infty$-stable and separable \cite{Kir06}. Let $\iota_{\O_\infty  \otimes B}:\O_\infty  \otimes B \hookrightarrow B_\omega$ be the embedding obtained by combining $\iota_{\O_\infty}$ and the canonical embedding  $B \hookrightarrow B_\omega$. Fix a positive contraction $x \in \O_\infty$ with full spectrum. Let $k \in B_\omega \cap B'$ be a lift of $\iota_{\O_\infty}(x)$ and say $k=(k_i)_i$ where $k_i \in B$ are positive contractions.

Let $(h_i)_i$ be a quasicentral approximate unit for the extension with $h_{i+1}h_i = h_i$ for all $i \in \N$. As in the proof of \cite[Theorem 1]{Gla20}, by passing to a subsequence, we may assume without loss of generality that
\begin{align}\label{eqn:subsequence3}
    \|h_i\mu(bk_i)-\mu(bk_i)h_i\| &\rightarrow 0, \nonumber \\
    \|h_{i+1}\mu(bk_i)-\mu(bk_i)h_{i+1}\| &\rightarrow 0,\nonumber \\
    \|(1-h_{i+1})[\mu(b)\mu(b'k_i)-\mu(bb'k_i)]\| &\rightarrow 0,\nonumber \\
    \|(1-h_{i+1})[\mu(bk_i)\mu(b'k_i)-\mu(bk_i b'k_i)]\| &\rightarrow 0 
 \end{align}
for all $b, b' \in B$. 

We now define sequences of c.p.c.\ maps $\alpha_i:A \rightarrow A$, $\theta^{(0)}_i:B \rightarrow A$, and $\theta^{(1)}_i:B \rightarrow A$ by setting
 \begin{align}\label{eqn:maps}
     \alpha_i(a)&= h_i^{\frac{1}{2}} a h_i^{\frac{1}{2}}, \nonumber\\
    \theta^{(0)}_i(b)&= (h_{i+1}-h_i)^{\frac{1}{2}}\mu(b)(h_{i+1}-h_i)^{\frac{1}{2}}\nonumber\\
    &\quad\quad+(1-h_{i+1})^{\frac{1}{2}}\mu(k_i^{\frac{1}{2}}b k_i^{\frac{1}{2}})(1-h_{i+1})^{\frac{1}{2}}, \notag \\
    \theta^{(1)}_i(b)&=(1-h_{i+1})^{\frac{1}{2}}\mu((1-k_i)^{\frac{1}{2}}b (1-k_i)^{\frac{1}{2}})(1-h_{i+1})^{\frac{1}{2}},
 \end{align}
for $a\in A$ and $b\in B$. Note that we do not assume that $A$ and $B$ are unital: the occurrences of units in \eqref{eqn:maps} can be taken to be in the minimal unitisations of $A$ and $B$.	
By the same calculations as in the proof of \cite[Theorem 1]{Gla20}, the induced maps $\alpha:A \rightarrow A_\omega$, $\theta^{(0)}:B\rightarrow A_\omega$ and $\theta^{(1)}:B\rightarrow A_\omega$ are all order zero, and the sum $\alpha + \theta^{(0)} \circ \pi + \theta^{(1)} \circ \pi$ coincides with the canonical embedding $A \rightarrow A_\omega$ via constant sequences.

Moreover, for $b \in B$, we compute that 
\begin{align}\label{eqn:pi-omega-1}
\pi_\omega(\theta^{(0)}(b)) &= \pi_\omega(\mu_\omega(k^{\tfrac{1}{2}}bk^{\tfrac{1}{2}})) \notag\\
&= \pi_\omega(\mu_\omega(kb)) \notag\\ 
&= kb \notag\\ 
&= \iota_{\O_\infty  \otimes B}(x \otimes b),
\end{align}
where we use that $h_i, h_{i+1} \in J$ in the first step, that $k \in B_\omega \cap B'$ in the second step, and that $\mu$ is a splitting of $\pi$ in the third step. Similarly, $\pi_\omega(\theta^{(1)}(b)) = b-kb = \iota_{\O_\infty  \otimes B}((1-x) \otimes b)$.

As $\dimnuc(J) < \infty$, $J$ has the corona factorisation property by Proposition \ref{prop:NucDimCFP}. Hence, by Proposition \ref{prop:CFPpurelylarge}, the extension $0\rightarrow J \rightarrow A \rightarrow B\rightarrow 0$ is purely large. (Alternatively, note that $B$ is purely infinite and use Proposition \ref{prop:PureInf-PureLarge}.) For each $i \in \N$, define hereditary subalgebras $A_i = \overline{(1-h_{i+1})A(1-h_{i+1})} \subseteq A$ and $J_i = \overline{(1-h_{i+1})J(1-h_{i+1})} \subseteq J$. By Proposition \ref{prop:cutdown}, each extension $0\rightarrow J_i \rightarrow A_i \rightarrow B\rightarrow 0$ is purely large. 

Set $X = \prim(B)$. View $B$ as an $X$-$\rm{C}^*$-algebra with the canonical $X$-action. Let $C_0(0,1] \otimes B$ and $B_\omega$ have the induced $X$-actions. Equip $A_\omega$ with the induced $X$-action obtained by pulling back ideals via $\pi_\omega:A_\omega \rightarrow B_\omega$, and equip $\prod_{i\to\omega} A_i$ with the induced $X$-action obtained by pulling back ideals via the restriction of $\pi_\omega$.

Let $\Theta^{(0)}:C_0(0,1] \otimes B \rightarrow A_\omega$ be the $^*$-homomorphism corresponding to the c.p.c.\ order zero map $\theta^{(0)}$. Noting that the image of $\theta^{(1)}_i$ is contained in $A_i$, we view $\theta^{(1)}$ as a c.p.c.\ order zero map $B \rightarrow \prod_{i\to\omega} A_i$ and let $\Theta^{(1)}:C_0(0,1] \otimes B \rightarrow \prod_{i\to\omega} A_i$ be the corresponding $^*$-homomorphism.

The goal is now to prove that $\Theta^{(0)}$ and $\Theta^{(1)}$ are $X$-full using Propositions \ref{prop:fullness} and \ref{prop:fullness6AM}. This will require some preliminary computations.
By \eqref{eqn:pi-omega-1}, we have $\pi_\omega \circ \Theta^{(0)}(\mathrm{id}_{(0,1]} \otimes b) = \iota_{\O_\infty  \otimes B}(x \otimes b)$ for all $b \in B$. Since both $\pi_\omega \circ \Theta^{(0)}$ and $\iota_{\O_\infty  \otimes B}$ are $^*$-homomorphisms, we deduce that  $\pi_\omega \circ \Theta^{(0)}((\mathrm{id}_{(0,1]})^n \otimes b^n) = \iota_{\O_\infty  \otimes B}(x^n \otimes b^n)$ for all $b \in B$, from which we deduce that
\begin{align}\label{eqn:pi-omega-2}
\pi_\omega \circ \Theta^{(0)}(f \otimes b) &= \iota_{\O_\infty  \otimes B}(f(x) \otimes b) = f(k)b
\end{align}
for all $f \in C_0(0,1]$ and all $b \in B$. In a similar manner, we deduce that $\pi_\omega \circ \Theta^{(1)}(f \otimes b) = \iota_{\O_\infty  \otimes B}(f(1-x) \otimes b)$ for all $f \in C_0(0,1]$ and all $b \in B$.

Let $I_{C_0(0,1]} \lhd C_0(0,1]$ and $I_B \lhd B$ be non-zero ideals. Let $f \in I_{C_0(0,1]}$ be a non-zero positive contraction and let $b \in I_B$. As $x$ is a positive contraction of full spectrum, then $f(x) \neq 0$. Since $x \in \O_\infty$ and $\O_\infty$ is simple and purely infinite, there exists $y \in \O_\infty$ such that $yf(x)y^* = 1_{\O_\infty}$. For any approximate identity $(z_j)_j$ for $B$, we have 
\begin{equation}\label{eqn:get-to-1-otimes-b}
	 1_{O_\infty} \otimes b = \lim_{j \to \infty} (y \otimes z_j)(f(x) \otimes b)(y^* \otimes z_j). 
\end{equation}
Applying the $^*$-homomorphism $\iota_{\O_\infty  \otimes B}$ to \eqref{eqn:get-to-1-otimes-b} and using \eqref{eqn:pi-omega-2} 
\begin{align}
b &= \lim_{j \to \infty} \iota_{\O_\infty  \otimes B}(y \otimes z_j) \; \iota_{\O_\infty  \otimes B}(f(x) \otimes b) \; \iota_{\O_\infty  \otimes B}(y^* \otimes z_j) \notag \\
&= \lim_{j \to \infty} \iota_{\O_\infty  \otimes B}(y \otimes z_j) \; \left(\pi_\omega \circ \Theta^{(0)}(f \otimes b) \right)\; \iota_{\O_\infty  \otimes B}(y^* \otimes z_j).
\end{align}
Hence, $b \in \I(\pi_\omega \circ \Theta^{(0)})(I_{C_0(0,1]} \otimes I_B)$. Since $b \in I_B$ was arbitrary, it follows that $\I(\pi_\omega \circ \Theta^{(0)})(I_{C_0(0,1]} \otimes I_B) \supseteq \overline{B_\omega I_B B_\omega}$. The reverse inclusion follows from \eqref{eqn:pi-omega-2}, so $\I(\pi_\omega \circ \Theta^{(0)})(I_{C_0(0,1]} \otimes I_B) = \overline{B_\omega I_B B_\omega}$. The same argument applied to $\pi_\omega \circ \Theta^{(1)}$, using that $1-x$ is also a positive contraction of full spectrum, gives $\I(\pi_\omega \circ \Theta^{(1)})(I_{C_0(0,1]} \otimes I_B) = \overline{B_\omega I_B B_\omega}$. By Proposition \ref{prop:fullness}, $\pi_\omega \circ \Theta^{(0)}$ and $\pi_\omega \circ \Theta^{(1)}$ are $X$-full. Therefore, the maps $\Theta^{(0)}$ and $\Theta^{(1)}$ are $X$-full by Proposition \ref{prop:fullness6AM}.

Fix finite dimensional approximations $\phi_i = \eta_i \circ \psi_i$ for $B$ as given by Lemma \ref{lem:approximations}, and adopt the notation of Lemma \ref{lem:approximations} hereinafter. The map $\Phi:C_0(0,1] \otimes B \rightarrow B_\omega$ is $X$-full by Proposition \ref{prop:fullness}. For each $i \in \N$, there exists a c.p.c.\ order zero lift $\hat{\eta}_i:F_i \rightarrow A_i$ of the c.p.c.\ order zero map $\eta_i:F_i \rightarrow B$ by Proposition \ref{prop:oz-lifting}. Set $\hat{\phi}_i = \hat{\eta}_i \circ \psi_i$ and let $\hat{\phi}:B \rightarrow \prod_{i\to\omega} A_i$ be the induced map. Write $\hat{\Phi}:C_0(0,1] \otimes B \rightarrow \prod_{i\to\omega} A_i$ for the corresponding $^*$-homomorphism from the cone. Since $\pi_\omega \circ \Phi = \hat{\Phi}$, we see that $\hat{\Phi}$ is $X$-full by Proposition \ref{prop:fullness6AM}. 

By Proposition \ref{thm:OinfClassification}, there exist unitaries $u \in \M_2(A_\omega)^\sim$ and $v \in \M_2(\prod_{i\to\omega} A_i)^\sim$ such that 
\begin{align}\label{eqn:unitary-twist}
\operatorname{Ad}(u)\circ (\hat{\Phi} \oplus 0)&=\Theta^{(0)}\oplus 0,\nonumber\\
\operatorname{Ad}(v)\circ (\hat{\Phi} \oplus 0)&=\Theta^{(1)}\oplus 0.
\end{align}
In particular, $\operatorname{Ad}(u)\circ(\hat{\phi} \oplus 0)=\theta^{(0)}\oplus 0$ and  $\operatorname{Ad}(v)\circ (\hat{\phi} \oplus 0)=\theta^{(1)}\oplus 0$. Write $u = (u_i)$ and $v = (v_i)$ for unitaries $u_i \in \M_2(A)^\sim$ and $v_i \in \M_2(A_i)^\sim$.

Since $\overline{h_iAh_i}$ is a hereditary subalgebra of $J$, we have $\dimnuc(\overline{h_iAh_i}) = \dimnuc(J) = n$ by \cite[Proposition 2.5]{WZ10}. Therefore, there exist finite-dimensional algebras $G_i^{(j)}$, c.p.c.\ maps $\beta_i:\overline{h_iAh_i} \rightarrow G_i^{(0)} \oplus \ldots \oplus  G_i^{(n)}$, and c.p.\ maps $\gamma_i: G_i^{(0)} \oplus \ldots \oplus  G_i^{(n)} \rightarrow \overline{h_iAh_i}$ with every restriction $\gamma_i|_{G_i^{(j)}}$ c.p.c. order zero such that $\|\alpha_i(a) - \gamma_i(\beta_i(\alpha_i(a)))\| \rightarrow 0$ for all $a \in A$, using Definition \ref{def:DimNuc} and the fact that $A$ is separable.

We now define completely positive approximations
\begin{equation}
\begin{tikzpicture}[node distance = 8cm, auto, baseline=(current  bounding  box.center)]
\node (B) {$A$};
\node(temp) [node distance = 4cm, right of = B] {};
\node (prod)[right of = B] {$\M_2(A)$};
\node (CB)[node distance = 2cm, below of = temp] {$F_i \oplus F_i \oplus G_i^{(0)} \oplus \cdots  \oplus G_i^{(n)}$,};

\draw[->] (B) to node[swap] {$\rho_i$} (CB);
\draw[->] (CB) to node[swap] {$\chi_i$} (prod);
\draw[->] (B) to node {$\operatorname{id} \oplus 0$} (prod);
\end{tikzpicture}
\end{equation} 
where 
\begin{align}
	\rho_i(a) &= (\psi_i(\pi(a)), \psi_i(\pi(a)), \beta_i(\alpha_i(a)))\\
	\chi_i(x,y,z) &= \operatorname{Ad}(u_i)(\hat{\eta}_i(x) \oplus 0) + \operatorname{Ad}(v_i)(\hat{\eta}_i(y) \oplus 0) + \gamma_i(z) \oplus 0. \nonumber
\end{align}
for $a \in A$, $x,y \in F_i$ and $z \in G_i^{(0)} \oplus \cdots  \oplus G_i^{(n)}$.

By \eqref{eqn:unitary-twist} and the fact that $\alpha + \theta^{(0)} \circ \pi + \theta^{(1)} \circ \pi$ coincides with the canonical embedding $A \rightarrow A_\omega$, we have $\lim_{i \to \omega}\chi_i(\rho_i(a)) = a$ for all $a \in A$. The restriction of $\chi_i$ to each summand is order zero. Moreover, as $\overline{h_iAh_i}$ is always orthogonal to $A_i$, we have that $\chi_i$ restricted to $(F_i \oplus G_i^{(0)})$ is order zero. Therefore, $\operatorname{id}_A \oplus \, 0$ has nuclear dimension at most $n+1$.
 
As $A$ is a hereditary subalgebra of $\M_2(A)$ and $\mathrm{id}_A$ is the co-restriction of $\operatorname{id}_A\oplus\ 0$ to $A$, we have that $\mathrm{id}_A$ has nuclear dimension at most $n+1$ by \cite[Proposition 1.6]{BGSW19}. Hence, $\dim_{\mathrm{nuc}}(A)\leq n+1$. This completes the proof of the upper bound. 
When $J$ is an AF algebra, $\dimnuc(J) = 0$. Hence,  $1 \leq \dim_{\mathrm{nuc}}(A)\leq 1$. Therefore, $\dim_{\mathrm{nuc}}(A) = 1$.
\end{proof}

\end{document}